\documentclass[12pt]{amsart}
\usepackage{verbatim,amscd,amssymb}
\usepackage[active]{srcltx}
\usepackage[all,cmtip]{xy}
\newtheorem{thm}{Theorem}[section]
\newtheorem{prop}[thm]{Proposition}
\newtheorem{lem}[thm]{Lemma}

\newtheorem*{Mthm}{Main Theorem}

\theoremstyle{remark}

\newtheorem{rem}[thm]{Remark}

  \def\Cc{\mathcal{C}}
\def\C{\mathbb{C}}   
\def\Gc{\mathcal{G}} 

\def\N{\mathbb{N}}

\def\Rc{\mathcal{R}}

\renewcommand\emptyset{\varnothing}
\def\eps{\varepsilon}

\def\CP{\mathbb{C}P}

\title{Partial holomorphic semiconjugacies between rational functions}
\author{V. Timorin}
\address{
Faculty of Mathematics\\
Higher School of Economics\\
7 Vavilova St. 112312\\
Moscow, Russia}
\email{vtimorin@hse.ru}
\thanks{Partially supported by the Deligne scholarship, RFBR grant 10-01-00540-a
and MPIM research grant}

\begin{document}

\begin{abstract}
We establish a general result on the existence of partially defined semiconjugacies
between rational functions acting on the Riemann sphere.
The semiconjugacies are defined on the complements to at most one-dimensional sets.
They are holomorphic in a certain sense.
\end{abstract}

\maketitle

\section{Introduction}
Let $A\subseteq\CP^1$ be a $G_\delta$-set, i.e. the intersection of countably many open sets.
A map $\Phi:A\to\CP^1$ is said to be {\em holomorphic} if
there is a sequence of holomorphic maps $\Phi_n:A_n\to\CP^1$ such that
$A_n\supseteq A$ are open subsets of $\CP^1$, and $\Phi_n$
converge to $\Phi$ uniformly on $A$.
Recall that a {\em real semi-algebraic subset} in a real algebraic variety
is a set given by any boolean combination of real algebraic equations and inequalities.
The main result of this paper is the following

\begin{Mthm}
Suppose that $R:\CP^1\to\CP^1$ is a hyperbolic rational function with a finite postcritical set $P_R$,
and $Q:\CP^1\to\CP^1$ is a rational function such that the diagram
$$
 \begin{CD}
   \CP^1 @>R>> \CP^1\\
   @V \tilde\eta VV @VV\eta V\\
   \CP^1 @>Q>> \CP^1
  \end{CD}
$$
is commutative, where $\eta$ and $\tilde\eta$ are homeomorphisms that coincide on $R(P_R)$
and are isotopic relative to $R(P_R)$.
If $P_R$ has at least three points,
then there exists a countable union $Z$ of real semi-algebraic sets of
codimension $>0$ backward invariant under
$Q$ and a holomorphic map $\Phi:\CP^1-Z\to\CP^1$ such that $R\circ\Phi=\Phi\circ Q$ on $\CP^1-Z$.
\end{Mthm}

Recall that the postcritical set of a rational function $R:\CP^1\to\CP^1$ is defined
as the closure of the set $\{R^{\circ n}(c)\}$, where $c$ runs through
all critical points of $R$, and $n$ runs through all positive integers.
A rational function $R$ is called {\em critically finite} if its postcritical
set is finite.
The relation between $R$ and $Q$ resembles Thurston equivalence
(in which we require that $\eta$ and $\tilde\eta$ coincide on $P_R$,
map $P_R$ onto the postcritical set of $Q$, and be isotopic relative to $P_R$)
but is in fact much weaker.
If $Q$ has at least one superattracting cycle of period $>1$, then, as a rule, there are
infinitely many different functions $R$ that satisfy the assumptions of the theorem.
Note that the assumptions of the Main Theorem imply the existence of at
least one super-attracting cycle of $Q$, namely, $\eta(C)$, where $C$
is a super-attracting cycle of $R$ in $R(P_R)$.

The map $\Phi$ from the Main Theorem semiconjugates the restriction of $Q$ to $\CP^1-Z$
with a certain restriction of $R$.
Note that the set $\CP^1-Z$ is forward invariant under $Q$ since $Z$ is backward invariant.
The set $Z$ can be constructed explicitly, and in many different ways.
The Main Theorem is only useful in combination with the knowledge of what $Z$ is.
In fact, $Z$ is very flexible and can be tailored to specific needs.
Semi-algebraicity is only one possible application of this flexibility.
We could have replaced semi-algebraicity with many other nice properties.
The map $\Phi$ is holomorphic, in particular, continuous.
Being holomorphic gives more information than continuity although many nice properties
of holomorphic maps fail in our setting, e.g. the uniqueness theorem.
Note, however, that the restriction of $\Phi$ to the interior of
$\CP^1-Z$ is holomorphic in the usual sense.

The main theorem is closely related to the {\em regluing surgery} of \cite{T},
although we will not use it explicitly.
Many of the ideas used in this paper are inspired by works of M. Rees (see
e.g. \cite{R}).
In Section \ref{s:supp}, we briefly describe some particular applications of the Main Theorem,
from which these relations may become clear.
To prove the Main Theorem, we will use a version of Thurston's algorithm \cite{DH93}.

A bigger part of this work has been done during my visit
at Max Planck Institute for Mathematics, Bonn (January--April 2010).
I would like to thank the institute for providing inspiring working conditions.

\section{Supportive real semi-algebraic sets}
\label{s:supp}

Recall that the {\em support} of a homeomorphism $\sigma:\CP^1\to\CP^1$
is defined as the closure of the set of all points $x\in\CP^1$
such that $\sigma(x)\ne x$.
Let $P$ be a finite subset of $\CP^1$ and $\sigma:\CP^1\to\CP^1$ a homeomorphism.
We say that a closed subset $Z_0\subset\CP^1$ is {\em supportive}
for $(\sigma,P)$ if, for every open neighborhood $U$ of $Z_0$, there
exists a homeomorphism $\tilde\sigma$ with the following properties:
\begin{itemize}
  \item the homeomorphisms $\sigma$ and $\tilde\sigma$ coincide on $P$;
  \item they are isotopic relative to $P$;
  \item the support of $\tilde\sigma$ is contained in $U$.
\end{itemize}

\begin{prop}
  \label{P:supp}
  For every orientation-preserving homeomorphism $\sigma:\CP^1\to\CP^1$ and every finite set
  $P\subset\CP^1$, there exists a closed real semi-algebraic set of
  positive codimension supportive for $(\sigma,P)$.
\end{prop}

\begin{proof}
  Consider a continuous one-parameter family $\sigma_t$ of homeomorphisms
  connecting $\sigma$ with the identity: $\sigma_0=\sigma$, $\sigma_1=id$.
  For every point $x\in P$ define a continuous path $\beta_x:[0,1]\to\CP^1$
  by the formula $\beta_x(t)=\sigma_t(x)$.
  We can assume that the curves $\beta_x[0,1]$ are real semi-algebraic
  (e.g. by the Weierstrass approximation theorem) and
  the paths $\beta_x$ are smooth.

  Take any open neighborhood $U$ of the closed real semi-algebraic set
  $$
  Z_0=\bigcup_{x\in P}\beta_x[0,1].
  $$
  Define a vector field $v_t$ on $\CP^1$ depending smoothly on $t$
  and having the following properties:
  \begin{itemize}
    \item at the point $\beta_x(t)$, the vector field $v_t$ is equal to $d\beta_x(t)/dt$;
    \item $v_t=0$ outside of a small neighborhood of $\beta_x(t)$
    contained in $U$.
  \end{itemize}
  Let now $g^t$ be the time $[0,t]$ flow of the non-autonomous
  differential equation $\dot z(t)=v_t$.
  Clearly, $g^t(\sigma(x))=\beta_x(t)$ for every $x\in P$, and the support of $g^t$ is contained in $U$.

  Now define the following isotopy:
  $\tilde\sigma_t=(g^t)^{-1}\circ\sigma_{t}$, $t\in [0,1]$.
  We have $\tilde\sigma_0(x)=\sigma(x)$.
  On the other hand, the support of $\tilde\sigma_1=(g_1)^{-1}$ is contained in $U$.
  Therefore, $Z_0$ is supportive for $(\sigma,P)$.
\end{proof}

\begin{rem}
\label{R:supp}
  It also follows from the proof of Proposition \ref{R:supp} that
  there exists a continuous one-parameter family of homeomorphisms
  connecting $\tilde\sigma_1$ with the identity such that the supports of
  all these homeomorphisms are contained in $U$.
  To obtain such a family, we can apply the procedure, described in the
  proof of Proposition \ref{R:supp}, to homeomorphisms $\sigma_t$ rather than $\sigma$.
\end{rem}

\begin{prop}
\label{P:Z0}
  Let rational functions $Q$ and $R$ be as in the statement of the Main Theorem.
  Set $\sigma=\tilde\eta\circ\eta^{-1}$, so that $\sigma\circ Q=\tilde\eta\circ R\circ\tilde\eta^{-1}$.
  There exists a closed real semi-algebraic set $Z_0$ of positive codimension
  that is supportive for $(\sigma,\eta(P_R))$ and that is disjoint from $\eta(R(P_R))$.
\end{prop}

\begin{proof}
  Note that $\sigma=id$ on $\eta(R(P_R))$, and $\sigma$ is homotopic to the
  identity relative to the set $\eta(R(P_R))$, by the assumptions of the Main Theorem.
  In the proof of Proposition \ref{P:supp}, we can therefore assume that
  $\sigma_t(x)=x$ for all $x\in\eta(R(P_R))$ and all $t\in [0,1]$.
  For these $x$, we do not consider the curves $\beta_x$.
  The rest of the proof works as before.
\end{proof}

The main theorem will follow from

\begin{thm}
\label{T:holmodZ}
Suppose that $Q$ and $R$ are rational functions as in the statement of the
Main Theorem, and $Z_0$ is the set from Proposition \ref{P:Z0}.
Define the set
$$
Z=\bigcup_{n=1}^\infty Q^{-n}(Z_0).
$$
There exists a holomorphic map $\Phi:\CP^1-Z\to\CP^1$ such that
$R\circ\Phi=\Phi\circ Q$ on $\CP^1-Z$.
\end{thm}

Clearly, $Z$ is a countable union of real semi-algebraic sets of positive
codimension, namely, of the iterated preimages of $Z_0$ under $Q$.

We now describe some particular applications of Theorem \ref{T:holmodZ}.
Let $\Rc_2$ be the set of M\"obius conjugacy classes of quadratic rational functions
with marked critical points.
Following M. Rees \cite{R} and J. Milnor \cite{M}, consider the slice $Per_k(0)\subset\Rc_2$
defined by the condition that the second critical point is periodic of period $k$.
The slices $Per_k(0)$ form a natural sequence of parameter curves starting with $Per_1(0)$,
the plane of quadratic polynomials.
We say that a critically finite rational function $R\in Per_k(0)$ is of {\em type C}
if the first critical point is eventually mapped to the second (periodic) critical point
but does not belong to the cycle of the second critical point.
Let $R\in Per_k(0)$ be any type C critically finite rational function, and $Q\in Per_k(0)$
be almost any function.
There are only finitely many exceptions, and all exceptional maps are critically finite.
Then, as M.Rees has shown in \cite{R}, there is a homeomorphism $\sigma_\beta:\CP^1\to\CP^1$,
whose support is contained in an arbitrarily small neighborhood of a simple path $\beta:[0,1]\to\CP^1$
such that $\sigma_\beta(\beta(0))=\beta(1)$, and
$\sigma_\beta\circ Q$ is a critically finite branched covering Thurston equivalent to $R$.
It is a simple exercise to check that $Q$ and $R$ satisfy the assumptions of the Main Theorem,
and that we can take $Z_0=\beta[0,1]$.
In this way, we obtain a partial semiconjugacy between almost any function from $Per_k(0)$
and any type C critically finite function.
It will be defined on the complement to all pullbacks of the simple curve $Z_0$ under $Q$.
E.g. for $Q$ we can take a quadratic polynomial $z\mapsto z^2+c$, whose critical point $0$
is periodic of period $k$.
Then the Main Theorem implies, in particular, the topological models for {\em captures} of $Q$
introduced in \cite{R}.

\section{Thurston's algorithm}
\label{s:Thurston}

In the proof of Theorem \ref{T:holmodZ},
we will use {\em Thurston's algorithm} (see \cite{DH93}).
We now briefly recall how it works (in a slightly more general setting than usual).
Let $X$ be a topological space and $f:X\to X$ be a continuous map.
Suppose that there is a topological semiconjugacy between $f$
and a rational function acting on the Riemann sphere.
Thurston's algorithm serves to find this semi-conjugacy.
It starts with a surjective continuous map $\phi_0:X\to\CP^1$.
Assume that there is a rational function $R_0$ and a continuous map
$\phi_1:X\to\CP^1$ that make the following diagram commutative:
  $$
  \begin{CD}
   X @>f>> X\\
   @V \phi_1 VV @VV\phi_0 V\\
   \CP^1 @>R_0>> \CP^1
  \end{CD}
  $$
This is always the case if the map $\phi_0$ is a homeomorphism
(in particular, $X$ is a topological sphere) and $f$ a branched covering.
Indeed, we can arrange that $f$ and $\phi_0$ be smooth by small deformations
preserving the critical values of $\phi_0\circ f$.
Then we consider the pullback $\kappa$ of the complex structure on $\CP^1$
under the map $\phi_0\circ f$.
We can integrate $\kappa$, i.e. there is a homeomorphism $\phi_1:X\to\CP^1$
taking the complex structure $\kappa$ on $X$
to the standard complex structure on $\CP^1$.
Clearly, $R_0=\phi_0\circ f\circ\phi_1^{-1}$ preserves the standard
complex structure, hence it is a rational function.
If $f$ or $\phi_0$ were not smooth, then $R_0$ constructed for
smooth deformations of $f$ and $\phi_0$ will also work for $f$ and $\phi_0$,
i.e. $\phi_1$ can be defined as a branch of $R_0^{-1}(\phi_0\circ f)$.
Note that $R_0$ is only defined up to precomposition with an automorphism
of $\CP^1$, and $\phi_1$ is only defined up to post-composition with an
automorphism of $\CP^1$.

The transition from $\phi_0$ to $\phi_1$ is the main step of Thurston's algorithm.
Doing this step repeatedly, we obtain a sequence of maps $\phi_n$.
We want that $\phi_n$ converge to a semiconjugacy between $f$ and some rational function.

We will now use notation from Proposition \ref{P:Z0} and Theorem \ref{T:holmodZ}.
Let us consider Thurston's algorithm for $\tilde\sigma\circ Q$,
where $\tilde\sigma$ is a homeomorphism  isotopic to $\sigma$ relative to the set
$\eta(P_R)$.
Note that the branched covering $\tilde\sigma\circ Q$ is Thurston equivalent to $R$,
and $P=\tilde\eta(P_R)$ is the postcritical set of this branched covering.
Indeed, all critical values of $Q$ are contained in $\eta(P_R)$;
the images of these critical values under $\tilde\sigma$ are contained in
$\tilde\eta(P_R)=\tilde\sigma\circ\eta(P_R)$; the further images under $\tilde\sigma\circ Q$
are contained in $P'=\eta(R(P_R))$ because the action of $\tilde\sigma\circ Q$
on $\tilde\eta(P_R)$ coincides with the action of $Q$, and $Q(\tilde\eta(P_R))=P'$.
By Proposition \ref{P:Z0}, we can assume that the support of $\tilde\sigma$
is contained in an arbitrarily small neighborhood $U$ of $Z_0$.
We choose this neighborhood so that it is disjoint from the set $P'$.

Set $\hat f=\tilde\sigma\circ Q$.
Thurston's algorithm for $\hat f$ yields an infinite commutative diagram
$$
  \begin{CD}
   @> >> \CP^1 @>\hat f>> \CP^1 @>\hat f>>\CP^1 @>\hat f>> \CP^1\\
   \dots @. @V\hat\phi_3 VV @V\hat\phi_2 VV @V\hat\phi_1 VV @VV\hat\phi_0 V\\
   @> >> \CP^1 @>R_2>> \CP^1 @>R_1>>\CP^1 @>R_0>> \CP^1
  \end{CD}
  $$
We can set $\hat\phi_0=id$.
The classes of $\hat\phi_n$ in the Teichm\"uller space of $(\CP^1,P)$ are well defined.
They depend only on the Thurston equivalence class of $\hat f$ and not on a particular
choice of the homeomorphism $\tilde\sigma$.
However, the maps $\hat\phi_n$ are only defined up to post-composition with
conformal automorphisms of $\CP^1$.
To make a definite choice of $\hat\phi_n$, we introduce the following normalization.
Let $P_0$ be any 3-point subset of $P$.
Note that the sets $Q^{\circ n}(P_0)$ are disjoint from $U$ for all $n>0$
since they lie in $P'$.
We require that the restriction of every $\hat\phi_n$ to $P_0$ be the identity.
This normalization makes the maps $\hat\phi_n$ uniquely defined.
However, the maps $\hat\phi_n$ depend on the choice of $\tilde\sigma$.
The rational functions $R_n$ are uniquely defined by the classes of
$\hat\phi_n$ in the Teichm\"uller space of $(\CP^1,P)$ and
the normalization $\hat\phi_n|_{P_0}=id$.
Therefore, they do not depend on the choice of $\tilde\sigma$.

\begin{prop}
\label{P:indep}
Set $U_n$ to be the union of $Q^{-i}(U)$ for $i=1$, $\dots$, $n$.
The values of $\hat\phi_n$ at points $z\not\in U_n$ do not depend on a particular
choice of a homeomorphism $\tilde\sigma$ with support in $U$.
\end{prop}

\begin{proof}
Indeed, different homeomorphisms $\tilde\sigma_0$ and $\tilde\sigma_1$
are isotopic relative to $P$.
Let $\tilde\sigma_t$, $t\in [0,1]$ be an isotopy.
We can assume that the support of $\tilde\sigma_t$ is contained in $U$ for every $t$,
see Remark \ref{P:supp}.
Let $\hat\phi_{n,t}$ be the maps that correspond to $\hat\phi_n$
as we replace $\tilde\sigma$ with $\tilde\sigma_t$.
Take $z\not\in Q^{-n}(U)$.
Suppose by induction that $\hat\phi_{n-1,t}(\hat f(z))$ does not depend on $t$
(note that $\hat f(z)=Q(z)\not\in U_{n-1}$).
Then $\hat\phi_{n,t}(z)$ is a continuous path such that
$R_{n-1}\circ\hat\phi_{n,t}(z)=\hat\phi_{n-1}(z)\circ f$.
Hence the values of this path lie in the finite set $R_{n-1}^{-1}(\hat\phi_{n-1}(z)\circ f)$.
It follows that the path is constant.
\end{proof}

We can include the maps $\hat\phi_n$ into a continuous family of homeomorphisms
$\hat\phi_t:\CP^1\to\CP^1$ defined for all real non-negative values of $t$.
This is done in the following way.
By Remark \ref{R:supp}, there is a continuous one-parameter family of homeomorphisms $\tilde\sigma_t$,
$t\in [0,1]$ connecting $id$ with $\tilde\sigma$ such that
the supports of all $\tilde\sigma_t$ are contained in an arbitrarily small
neighborhood $U$ of $Z_0$, and $\tilde\sigma_t(\eta(P_R-P(P_R)))\subset Z_0$
(we use the notation of Proposition \ref{P:Z0}).
Consider the first step of Thurston's algorithm for $\hat f_t=\tilde\sigma_t\circ Q$:
$$
\begin{CD}
   \CP^1 @>\hat f_t>> \CP^1\\
   @V \hat\phi_t VV @VV id V\\
   \CP^1 @>R_{t-1}>> \CP^1
  \end{CD}
$$
Note that the critical values of $\hat f_t$ (hence also of $R_{t-1}$)
different from critical values of $Q$ lie in $Z_0$.
Normalize $\hat\phi_t$ by requiring that their restrictions to $P_0$ be the identity.
Note that, for $t=1$, we obtain the same $\hat\phi_1$ and $R_0$ as before.
We have $\hat f_1=\hat f$.
We can now start Thurston's algorithm with $\hat\phi_t$, $t\in (0,1)$ rather
than starting it with $\hat\phi_0=id$ (let me stress however that we do
{\em the same} algorithm for all $\hat\phi_t$, namely, the algorithm associated with $\hat f$!).
In this way, we obtain a continuous path of rational functions $R_t$
and a continuous path of homeomorphisms $\hat\phi_t$ defined for all real
nonnegative $t$ and satisfying the identity
$R_t\circ\hat\phi_{t+1}=\hat\phi_t\circ\hat f$ for $t\ge 0$.

\begin{prop}
  \label{P:R_n-conv}
  The rational functions $R_t$ converge to a rational function
  M\"obius conjugate to $R$.
\end{prop}

In the sequel, we will always assume that $R_t$ converge to $R$,
since nothing changes in the statement of the Main Theorem if we
replace the function $R$ by its M\"obious conjugate.

\begin{proof}
We know that $\hat f$ is Thurston equivalent to $R$.
From Thurston's Characterization Theorem \cite{DH93} (in fact, from its easy part)
it follows that the classes of $\hat\phi_n$ converge to the class of
some homeomorphism $\hat\phi_\infty:S^2\to\CP^1$ in the Teichm\"uller space of $(\CP^1,P)$.
The path $t\mapsto [\hat\phi_t]$, $t\in [0,\infty)$ converges in the Teichm\"uller space
of $(\CP^1,P)$ as well.
This follows from the convergence of $[\hat\phi_n]$ and the
contraction property of Thurston's pullback map.

Since the class $[\hat\phi_\infty]$ of $\hat\phi_\infty$ in the Teichm\"uller space
coincides with the class
$[M\circ\hat\phi_\infty]$ for every M\"obius transformation $M$, we can assume that
$\hat\phi_\infty=id$ on $P_0$.
Convergence of $[\hat\phi_t]$ to $[\hat\phi_\infty]$ means that there is a family of
of quasiconformal homeomorphisms $h_t:\CP^1\to\CP^1$ such that the quasiconformal constant
of $h_t$ tends to 1, and the equality $\hat\phi_t=h_t\circ\hat\phi_\infty$
holds on $P$ and holds on $\CP^1$ up to isotopy relative to $P$.
The maps $\hat\phi_t$ and $\hat\phi_\infty$ are the identity on $P_0$, hence so is $h_t$.
It follows that $h_t$ converge uniformly to the identity.

Note that the branched covering $h_t^{-1}\circ R_t\circ h_{t+1}$ is homotopic
to $\hat\phi_\infty\circ f\circ\hat\phi_\infty^{-1}$ relative to the set $\hat\phi_\infty(P)$ through
branched coverings.
Since $h_t\to id$, any partial limit of $R_t$ as $t\in\infty$
is a rational function homotopic to $\hat\phi_\infty\circ f\circ\hat\phi_\infty^{-1}$ relative to the set
$\hat\phi_\infty(P)$ through branched coverings
(in particular, this rational function is critically finite, hyperbolic and Thurston
equivalent to $\hat f$).
By Thurston's Uniqueness Theorem, such rational function is unique.
\end{proof}

Recall that, by our assumptions, $P\cap Z=\emptyset$.
Set $Z_t$ to be the union of $Q^{-i}(Z_0)$ for $i$ running from 1
to the smallest integer that is greater than or equal to $t$.
We will now define a family of holomorphic maps
$\Phi_t:\CP^1-Z_t\to\CP^1$
with the following properties:
$$
R_t\circ\Phi_{t+1}=\Phi_t\circ Q,\quad\Phi_t|_P=\hat\phi_t|_P,
$$
where the rational functions $R_t$ are the same as before.
For $z\not\in U_n$, where $n\ge t$, we set $\Phi_t(z)=\hat\phi_t(z)$.
This value is well-defined by the construction of $\hat\phi_t$ and the
same argument as in Proposition \ref{P:indep}.
On the other hand, for every $z\not\in Z_t$, we can choose $U$
such that $z\not\in U_n$, so that the definition applies.
The holomorphy of $\Phi_t$ follows from the fact that locally near
$z\not\in Z_t$,
the function $\Phi_t(z)$ is a branch of $R_{t-1}^{-1}\circ\Phi_t\circ Q$.

The key lemma is the following:

\begin{lem}
\label{L:unif}
  The maps $\Phi_n:\CP^1-Z\to\CP^1$, $n=1,2,\dots$, converge uniformly.
\end{lem}

\begin{proof}[Proof of Theorem \ref{T:holmodZ} assuming Lemma \ref{L:unif}]
  Every map $\Phi_n$ is a restriction of a holomorphic function defined on
the complement to a real semi-algebraic set of positive codimension.
By definition of a holomorphic function on $\CP^1-Z$, it follows that the uniform
limit $\Phi$ of $\Phi_n$ is holomorphic on $\CP^1-Z$.
\end{proof}

Lemma \ref{L:unif} is hard to approach directly because $\CP^1-Z$ is a bad space.
Therefore, we will consider a certain compactification of it.

Let $X_0$ be the Caratheodory compactification of $\CP^1-Q^{-1}(Z_0)$.
This is a compact real 2-dimensional manifold with boundary
that has a canonical projection $\pi_0$ onto $\CP^1$.
The space $X_0$ is not necessarily connected.
Define $X$ as the set of sequences $(x_n)$ in $X_0$, $n=1,2,\dots$ such that
$\pi_0(x_{n+1})=Q(\pi_0(x_n))$ for all $n$.
The topology on $X$ is induced from the direct product topology on
the space of all sequences.
Thus $X$ is a compact Hausdorff space.
Define a continuous self-map $f:X\to X$ as follows: if $x$ is a sequence
$(x_1,x_2,x_3,\dots)$, then $f(x)=y$, where $y$ is the sequence $(x_2,x_3,\dots)$.
There is also a natural projection $\pi:X\to\CP^1$ given by
the formula $\pi(x_1,x_2,\dots)=\pi_0(x_1)$.
This projection is a semiconjugacy between $f$ and $Q$.
Note that $X$ can be identified with the projective limit of
Caratheodory compactifications of $\CP^1-Z_n$.

Since $P\cap Z=\emptyset$, the map $\pi$ restricted to $\pi^{-1}(P)$ is one-to-one.
Set $P_f=\pi(P)$.
There exists a one-parameter family of continuous maps
$\phi_t:X\to\CP^1$ such that $\phi_t=\Phi_t\circ\pi$ on $\pi^{-1}(\CP^1-Z)$.
Indeed, every $\Phi_t$ extends to the Caratheodory compactification of $\CP^1-Z_t$.
The maps $\phi_t$ make the following diagram commutative:
$$
\begin{CD}
   (X,f^{-1}(P_f)) @>f>> (X,P_f)\\
   @V \phi_{t+1} VV @VV \phi_t V\\
   (\CP^1,R_t^{-1}(P_t)) @>R_{t}>> (\CP^1,P_t)
  \end{CD}
\eqno{(*)}
$$
where $P_t=\phi_t(P_f)$.
The restrictions of the maps $\hat\phi_t$ to $P_f$ converge.
This follows from the convergence of $\hat\phi_t$ in the Teichm\"uller space of
$(\CP^1,P)$ (more precisely, from the convergence of the projections of $\hat\phi_t$ to the
moduli space of $\CP^1-P$).
Hence, the restrictions of $\phi_t$ to $P_f$ converge as well.
It is clear that, for every critical value $v$ of $\hat f$,
the limit of $\hat\phi_t(v)$ is a critical value of $R$.
It follows that the limit of $\hat\phi_t(z)$ is in $P_R$ for every $z\in P$.
It also follows that the limit of $\phi_t(x)$ is in $P_R$ for every $x\in P_f$.

\begin{lem}
  \label{L:finvP}
  There exists a map $\iota:f^{-1}(P_f)\to R^{-1}(P_R)$
  such that the following diagram is commutative
  $$
 \begin{CD}
   f^{-1}(P_f) @>f>> P_f\\
   @V \iota VV @VV\iota V\\
   R^{-1}(P_R) @>R>> P_R
  \end{CD}
$$
  and the restrictions of $\phi_t$ to the set $f^{-1}(P_f)$ converge to $\iota$.
\end{lem}

\begin{proof}
Consider a point $x\in f^{-1}(P_f)$.
We have proved that the points $\phi_t\circ f(x)\in P_t$ converge to some point $a\in P_R$.
Let $t_n$ be any sequence such that $\phi_{t_n}(x)$ converges; denote the limit by $b$.
Passing to the limit in both sides of the equation
$R_{t_n-1}\circ\phi_{t_n}(x)=\phi_{t_n-1}\circ f(x)$, we obtain that $R(b)=a$.
It follows that the entire $\omega$-limit set of the family $\phi_t(x)$
is contained in the finite set $R^{-1}(a)$.
As the $\omega$-limit set is connected, this implies that $\phi_t(x)$ converges to $b$
as $t\to\infty$.
Set $\iota(x)=b$.

  The commutative diagram in the statement of the lemma is obtained
  by passing to the limit as $t\to\infty$ in the diagram $(*)$
and using that $\phi_t(x)$ converges to a point in $P_R$
for every $x\in P_f$.
\end{proof}

Lemma \ref{L:unif}, and hence also the Main Theorem, is now reduced to

\begin{thm}
  \label{T:unif}
  The maps $\phi_n:X\to\CP^1$ converge uniformly.
\end{thm}

If this holds, then the maps $\Phi_n=\phi_n\circ\pi^{-1}$ on $\CP^1-Z$ also
converge uniformly.
The remaining part of the paper contains the proof of Theorem \ref{T:unif}.
This is a statement about uniform convergence of Thurston's algorithm.
As such, it is perhaps not surprising, although we state it for a topological
space $X$ that is not $S^2$
(actually, we need nothing from the space $X$ except that it is locally compact
and that some neighborhood of $P_f$ in $X$ has a structure of a Riemann surface;
however, specific properties of $\phi_n$ will be used, e.g. that the restrictions
of $\phi_n$ to $P_f$ converge and that $\phi_n$ are holomorphic near $P_f$).
Thurston's algorithm is generally expected to converge uniformly,
and theorems to this effect have been proved in a variety of contexts.
E.g. a general theorem about uniform convergence of Thurston's
algorithm has appeared in \cite{CT}.
It deals with hyperbolic but not necessarily critically finite rational functions.
I am grateful to Tan Lei for showing me a draft of this work.

The underlying ideas of the proof of Theorem \ref{T:unif} can be traced back to \cite{DH84}.
Very roughly, it is an application of the contraction principle to a
certain lifting map on a certain functional space.
It is even possible to state a general theorem of this sort but many fine
details would make its statement too cumbersome.
The things are not complicated but they are not straightforward either.

\section{The space $\Cc$}

{\footnotesize Notation: for a topological space $X$ and a metric space $Y$,
we denote by $C(X,Y)$ the set of all continuous maps from
$X$ to $Y$. We will always equip this set with the partially defined
uniform metric (note that the uniform distance between two elements of
$C(X,Y)$ may well be infinite).}

\smallskip

In this section, we start the proof of Theorem \ref{T:unif}.
We first set up a suitable function space.
In the next section, we prove the convergence in this space.
Consider the hyperbolic critically finite rational function $R$ from the
statement of the Main Theorem.
Recall that $P_R$ denotes the postcritical set of $R$.

We will need a metric on $\CP^1-P_R$ with certain properties:

\begin{lem}[Expanding metric on $\CP^1-P_R$]
\label{L:exp-metric}
There exists a piecewise smooth metric on $\CP^1-P_R$ equal to a constant
multiple of $|d\xi|/|\xi|$ near every point $z\in P_R$ for some local
holomorphic coordinate $\xi$ with $\xi(z)=0$
and such that the map $R:\CP^1-R^{-1}(P_R)\to\CP^1-P_R$ is uniformly expanding
with respect to this metric.
\end{lem}

\begin{proof}
It follows from hyperbolicity that there exists a neighborhood $V_0$ of the
Julia set $J(R)$ of $R$ and a Riemannian metric $g_0=\sigma(z)|dz|$ on $V_0$
such that $R$ is uniformly expanding with respect to $g_0$ i.e.
$$
\sigma\circ R(z)|dR(z)|\ge E_0 \sigma(z)|dz|\eqno{(1)}
$$
for some $E_0>1$ and all $z\in V_0\cap R^{-1}(V_0)$.
We can assume that $V_0$ is bounded by smooth curves and that $V_0=R^{-1}(R(V_0))$.
Now extend the metric $g_0$ to the set $R(V_0)-V_0-P_R$ by the following formula:
$$
g_0(z)=E_0\cdot\max_{i}g(S_i(z)),
$$
or, equivalently,
$$
\sigma(z)=E_0\cdot\max_{i}\left\{\sigma(S_i(z))\left|\frac{dS_i(z)}{dz}\right|\right\},
$$
where $S_i(z)$ are all local branches of $R^{-1}$ near $z$.
They are well defined since all critical values of $R$ belong to $P_R$.
With this definition, inequality $(1)$ holds also in $R(V_0)-P_R$.
Using the same formula, we can extend the metric $g_0$ to $R^{\circ m}(V_0)-P_R$
for every $m$, hence to the complement of an arbitrarily small
neighborhood of $P_R$.
The extended metric is piecewise smooth and satisfies inequality $(1)$
provided that $g_0$ is defined at both $z$ and $R(z)$ .

Now let $V_1$ be a small neighborhood of $P_R$ such that every
component of $V_1$ is a Jordan domain containing exactly one point of $P_R$.
By B\"ottcher's theorem, there exists a holomorphic function $\xi:V_1\to\C$
with simple zeros at all points of $P_R$
such that $\xi\circ R(z)=\xi(z)^{\nu(z)}$, where $\nu$ is a
locally constant function on $V_1$ taking
its values in $\N$.
We can also assume that $\xi$ is a holomorphic coordinate on every component of $V_1$.
Note that $R$ multiplies the metric $|d\xi|/|\xi|$ on $V_1$ by $\nu(z)$.
Set
$$
g_1(z)=\lambda(z)\frac{|d\xi(z)|}{|\xi(z)|},
$$
where $\lambda$ is a locally constant function on $V_1$, which we define below.
It suffices to define $\lambda$ on $P_R$.
Set
$$
E_1(z)=\lim_{n\to\infty}\left(\prod_{i=0}^{n-1}\nu(R^{\circ i}(z))\right)^{1/n}.
$$
This number is equal to the geometric mean of $\nu$ over the cycle,
to which $z$ eventually maps.
In particular $E_1(z)>1$.
The function $\lambda$ on $P_R$ is now defined by the property
$$
\lambda(R(z))=\frac{E_1(z)\lambda(z)}{\nu(z)}.
$$
If we fix an arbitrary positive value of $\lambda$ at an arbitrarily chosen
point of each periodic cycle in $P_R$, then this condition defines $\lambda$
uniquely.
The metric $g_1$ on $V_1-P_R$ thus defined gets multiplied by $E_1(z)$ under the map $R$.
Define the number $E_1>1$ as the minimum of $E_1(z)$ over all points in $P_R$.

We now want to combine the two metrics $g_0$ and $g_1$.
We can assume that $V_0\cup V_1=\CP^1-P_R$ and that both $V_0$ and $V_1$ are bounded by
smooth curves.
We can also assume that there is no point $z\in V_0-V_1$ such that
$R(z)\in V_1-V_0$
(so that every $R$-orbit that visits both $V_0$ and $V_1$
must enter the ``buffer zone'' $V_0\cap V_1$).
Set $g=\eps g_0$ on $V_0-V_1$, $g=g_1$ on $V_1-V_0$, and $g=\eps g_0+g_1$ on $V_0\cap V_1$.
As we will show, the map $R$ is uniformly expanding with respect to $g$ provided
that the number $\eps>0$ is small enough
so that e.g. $\eps g_0\le (\sqrt{E_1}-1)g_1$ everywhere on $V_0\cap V_1$.
Indeed, if $z\in V_0-V_1$ and $R(z)\in V_0\cap V_1$, then
$$
g(R(z))=\eps g_0(R(z))+g_1(R(z))\ge E_0\eps g_0(z)=E_0g(z).
$$
If $z\in V_0\cap V_1$ and $R(z)\in V_0\cap V_1$, then
$$
g(R(z))=\eps g_0(R(z))+g_1(R(z))\ge E_0\eps g_0(z)+E_1 g_1(z)\ge \tilde E g(z).
$$
where $\tilde E=\min(E_0,E_1)$.
Finally, if $z\in V_0\cap V_1$ and $R(z)\in V_1-V_0$, then
$$
g(R(z))=g_1(R(z))\ge E_1g_1(z)=(E_1-\sqrt{E_1})g_1(z)+\sqrt{E_1}g_1(z)\ge
$$
$$
\sqrt{E_1}\eps g_0(z)+\sqrt{E_1}g_1(z)=\sqrt{E_1}g(z).
$$
\end{proof}

In the sequel, we will write $Y$ for the space $\CP^1-P_R$ equipped with the
metric $g$ from Lemma \ref{L:exp-metric}.
Note that the metric $g$ is {\em proper}: every closed bounded set is compact.
It follows that $g$ is complete and locally compact.
Let $E>1$ be the expansion factor of $R$ with respect to the metric $g$.
In the notation of Lemma \ref{L:exp-metric}, we can set $E=\min(E_0,\sqrt{E_1})$.

We will use notation from Section \ref{s:Thurston}.
Note that there is an open neighborhood $O$ of the set $P_f$, on
which the map $\pi$ is one-to-one and such that $f(O)\subset O$.
We will assume that $O$ is sufficiently small.
The map $\pi$ defines a Riemann surface structure on $O$.
The maps $\phi_t$ are holomorphic on the set $O$ equipped with this structure.
Let $\Cc(O)$ denote the space of continuous maps $\chi:X\to\CP^1$
with the following properties:
\begin{enumerate}
  \item
  $\chi=\iota$ on $f^{-1}(P_f)$;
  \item
  $\chi^{-1}(P_R)\subseteq P_f$;
  \item
  $\chi^{-1}(R^{-1}(P_R))\subseteq f^{-1}(P_f)$;
  \item
  the restriction of $\chi$ to $O$ is holomorphic, and no point of $P_f$
  is a critical point of $\chi$.
\end{enumerate}
We will consider the following metric on $\Cc(O)$:
the distance between maps $\chi$ and $\chi^*\in\Cc(O)$ is the
uniform distance between the restrictions $\chi:X-P_f\to Y$ and
$\chi^*:X-P_f\to Y$ measured with respect to the metric $g$ on $Y$.
We need to prove that the distance between any two elements $\chi$ and $\chi^*$ of
$\Cc(O)$ is finite.
It suffices to make a local estimate near each point $x\in P_f$.
Let $W_x$ be a small neighborhood of $\iota(x)$, and
$\xi$ a holomorphic coordinate on $W_x$ such that $\xi(\iota(x))=0$.
Let $O_x$ be a small neighborhood of $x$ contained in $O$ and
such that $\chi(O_x)\subset W_x$ and $\chi^*(O_x)\subset W_x$.
Since both holomorphic functions $\xi\circ\chi$ and
$\xi\circ\chi^*$ have simple zeros at $x$, their ratio extends to
a holomorphic function on $O_x$ taking a nonzero value at $x$.
Note that the uniform distance between the maps $\chi:O_x-\{x\}\to Y$ and
$\chi^*:O_x-\{x\}\to Y$ in the metric $g$ is
$$
const\cdot\sup_{x'\in O_x-\{x\}}\left|\log\left(\frac{\xi\circ\chi(x')}{\xi\circ\chi^*(x')}\right)\right|
$$
for some local branch of the logarithm.
Indeed, the metric $g$ is equal to $const\cdot|d\log\xi|$ on $W_x$.
We see that the distance between $\chi$ and $\chi^*$ is finite.
A similar argument shows that the topology on $\Cc(O)$ coincides with the
topology induced from the uniform metric on $C(X,\CP^1)$.
Define $\Cc$ as the union of $\Cc(O)$ over all sufficiently small 
neighborhoods $O$ of $P_f$ such that $f(O)\subset O$. 
As a metric space, $\Cc$ is the inductive limit of the spaces $\Cc(O)$.

We will write $\tilde Y$ for $Y-R^{-1}(P_R)$.
Then $R:\tilde Y\to Y$ is a proper expansion with expansion factor $E$.
Being a proper map and a local homeomorphism, this map enjoys the
unique path lifting property.
This is a key to the following

\begin{lem}
  \label{L:lifting}
  If $\gamma:[0,1]\to\Cc(O)$ is a continuous path and $\tilde\chi_0\in\Cc(O)$
  is a map such that $R\circ\tilde\chi_0=\gamma(0)\circ f$, then
  there is a unique continuous path $\tilde\gamma:[0,1]\to\Cc(O)$ with
  the properties $\tilde\gamma(0)=\tilde\chi_0$ and
  $R\circ\tilde\gamma(t)=\gamma(t)\circ f$ for all $t\in [0,1]$.
\end{lem}

We will call the path $\tilde\gamma$ a {\em lift} of the path $\gamma$.

\begin{proof}
  For every $t$, consider the restriction of $\gamma(t)$ to $X-P_f$.
  We obtain a path $\gamma_*:[0,1]\to C(X-P_f,Y)$.

  Consider the map $\Gc:C(X-f^{-1}(P_f),\tilde Y)\to C(X-f^{-1}(P_f),Y)$
  given by the formula $\Gc(\chi)=R\circ\chi$.
  The map $R:\tilde Y\to Y$ is a proper expansion, and $X-f^{-1}(P_f)$
  is a locally compact space (as a complement to finitely many points
  in a compact Hausdorff space).
  It is a standard fact from topology (see e.g. Spanier \cite{S}) that
  in this case the map $\Gc$ has the path lifting property:
  given a path $\alpha:[0,1]\to C(X-f^{-1}(P_f),Y)$ and an element
  $\tilde\chi_0\in C(X-f^{-1}(P_f),\tilde Y)$ such that $R\circ\tilde\chi_0=\alpha(0)$,
  there exists a unique path $\tilde\alpha:[0,1]\to C(X-f^{-1}(P_f),\tilde Y)$
  such that $R\circ\tilde\alpha(t)=\alpha(t)$ for all $t\in [0,1]$
  and $\tilde\alpha(0)=\tilde\chi_0$.
  We take $\alpha(t)=\gamma_*(t)\circ f$ and consider the corresponding
  lift $\tilde\alpha$ (with $\tilde\chi_0$ as in the statement of the lemma).

  It is obvious that every map $\tilde\alpha(t)$ extends to a continuous map
  $\tilde\gamma(t)$ from $X$ to $\CP^1$ holomorphic on $O$.
  It remains to show that the maps $\tilde\gamma(t)$ belong to $\Cc(O)$.
  The following two properties imply this:
  \begin{enumerate}
    \item $\tilde\gamma(t)=\gamma(t)$ on $f^{-1}(P_f)$;
    \item $\tilde\gamma(t)^{-1}(R^{-1}(P_R))\subseteq f^{-1}(P_f)$.
  \end{enumerate}
  Note that both $\gamma(t)$ and $\tilde\gamma(t)$ restricted to $f^{-1}(P_f)$
  take values in $R^{-1}(P_R)$.
  For $\tilde\gamma(t)$, this follows from the defining
  identity $R\circ\tilde\gamma(t)=\gamma(t)\circ f$.
  Now property (1) holds by continuity (the two maps coincide for $t=0$ and take
  values in finite sets).
  To prove property (2), take any $x\in X$ such that $\tilde\gamma(t)(x)\in R^{-1}(P_R)$.
  Then
  $$
  \gamma(t)(f(x))=R\circ\tilde\gamma(t)(x)\in P_R,
  $$
  hence $f(x)\in P_f$, hence $x\in f^{-1}(P_f)$.
\end{proof}

\begin{rem}
  \label{R:lift-contr}
  Note that any lift of a rectifiable path in $\Cc$ is at least
  $E$ times shorter than the path itself.
  This follows from the fact that the map $\Gc$ is a local expansion
  with expansion factor $E$.
\end{rem}

\begin{prop}
There exists a continuous family of
homeomorphisms $\psi_t:\CP^1\to\CP^1$ defined for sufficiently large $t$
with the following properties:
\begin{itemize}
\item
$\psi_t\circ\phi_t\in\Cc$;
\item
$\psi_t\to id$ uniformly as $t\to\infty$;
\item
there is a neighborhood $W$ of $P_R$ such that the restrictions
of $\psi_t$ to $W$ are holomorphic;
\end{itemize}
\end{prop}

Note that a priori we cannot fix a neighborhood $O$ of $P_f$ such that
$\psi_t\circ\phi_t\in\Cc(O)$ for all $t$.
However, as can be seen from the proof, such neighborhood exists for
every bounded interval of values of $t$.

\begin{proof}
 Suppose that $t$ is sufficiently large so that $\phi_t(x)\ne\phi_t(x')$ for
$x$, $x'\in f^{-1}(P_f)$ unless $\iota(x)=\iota(x')$.
Recall that $P_t=\phi_t(P_f)$.
For $z\in R^{-1}_{t-1}(P_{t-1})$, we set $\psi_t(z)=\iota(x)$, where
$x\in f^{-1}(P_f)$ is any point such that $\phi_t(x)=z$.
By our assumption, the point $\psi_t(z)$ thus defined does not depend on the choice of $x$.
We have defined the map $\psi_t$ on $R^{-1}_{t-1}(P_{t-1})$.
It is clear that the map $\psi_t$ is injective on $R^{-1}_{t-1}(P_{t-1})$
(different points of $R^{-1}_{t-1}(P_{t-1})$ do not merge in the limit).
Since the sets $R^{-1}_{t-1}(P_{t-1})$ and $R^{-1}(P_R)$ have the same
cardinality, this map is actually a bijection between these two sets.

The uniform distance between the map $\psi_t$ on $R^{-1}_{t-1}(P_{t-1})$ and
the identity is bounded above by the uniform distance between $\phi_t$
on $f^{-1}(P_f)$ and $\iota$.
Hence this distance tends to $0$ as $t\to\infty$.
It follows that we can choose a continuous family $\psi_t$ so that
$\psi_t\to id$ as $t\to\infty$ and that $\psi_t$ are holomorphic on some
neighborhood of $P_R$ for all sufficiently large $t$.

It suffices to prove that $\psi_t\circ\phi_t$ belongs to $\Cc$.
By definition $\psi_t\circ\phi_t=\iota$ on $f^{-1}(P_f)$.
We need to check that:
\begin{enumerate}
\item $(\psi_t\circ\phi_t)^{-1}(P_R)\subseteq P_f$;
\item $(\psi_t\circ\phi_t)^{-1}(R^{-1}(P_R))\subseteq f^{-1}(P_f)$.
\end{enumerate}

To prove (1), suppose that $w=\psi_t\circ\phi_t(x)\in P_R$ for some $x\in X$.
Since $\psi_t$ is a homeomorphism taking $P_t$ to $P_R$,
and $P_R$ has the same cardinality as $P_t$,
we have $\psi_t^{-1}(w)\in P_t$.
Then $x\in\phi_t^{-1}(P_t)=P_f$.
Property (2) can be proved by the same argument.
\end{proof}

Theorem \ref{T:unif} (and hence the Main Theorem) reduces to the following:

\begin{thm}
\label{T:conv}
  The sequence $\chi_n=\psi_n\circ\phi_n$ converges in $\Cc$ as $n\to\infty$
through positive integers,
  hence in $C(X-P_f,Y)$ and in $C(X,\CP^1)$.
\end{thm}

Indeed, since $\psi_n$ converge to the identity,
we conclude that $\phi_n$ converge uniformly, q.e.d.
The convergence in $\Cc$ is perhaps a little surprising because
the space $\Cc$ is not complete.

\section{Contracting lifting}

In this section, we prove Theorem \ref{T:conv}, hence also the Main Theorem.
We will repeatedly use contraction properties of the lifting as
defined in Lemma \ref{L:lifting}.
Note that the map $\chi_t=\psi_t\circ\phi_t\in\Cc$ depends continuously on $t$ 
with respect to the uniform metric in $C(X,\CP^1)$.
It can also be arranged that, for every finite interval $[t_0,t_1]$,
there exists a neighborhood $O$ of $P_f$ such that $\chi_t\in\Cc(O)$
for all $t\in [t_0,t_1]$. 
Hence $\chi_t$ form also a continuous family in $\Cc$.

Recall that we have the following commutative diagram:
 $$
  \begin{CD}
  (X,f^{-1}(P_f)) @>f>> (X,P_f)\\
  @V \phi_{t+1}VV @VV \phi_t V\\
  (\CP^1,R_t^{-1}(P_t)) @>R_t>> (\CP^1,P_t)\\
  @V \psi_{t+1}VV @VV \psi_t V\\
  (\CP^1,R^{-1}(P_R)) @. (\CP^1,P_R)
  \end{CD}
 $$

The family $\psi_t$ can be thought of as a continuous path in $C(\CP^1,\CP^1)$ defined on
the compact interval $[0,\infty]$: it suffices to set $\psi_\infty=id$.
There exists a unqiue continuous path $t\mapsto\tilde\psi_t$, $t\in [1,\infty]$ such that
$\tilde\psi_\infty=id$ and $R\circ\tilde\psi_t=\psi_{t-1}\circ R_{t-1}$ for all $t\ge 1$.
We have $\tilde\psi_t=\psi_t$ on $R^{-1}_{t-1}(P_{t-1})$ by continuity,
since both maps are equal to the identity for $t=\infty$,
and both take values in $R^{-1}(P_R)$.
Since $\psi_t\circ\phi_t\in\Cc$ and $\tilde\psi_t=\psi_t$ on $R^{-1}_{t-1}(P_{t-1})$,
we also have $\tilde\chi_t=\tilde\psi_t\circ\phi_t\in\Cc$.
Note that the family $\tilde\chi_t$ satisfies the following identity:
$$
R\circ\tilde\chi_{t+1}=\chi_t\circ f.
$$
Indeed, we have
$$
R\circ\tilde\psi_{t+1}\circ\phi_t=\psi_{t}\circ R_{t}\circ\phi_{t+1}=\psi_{t}\circ\phi_t\circ f.
$$

Note that both $\psi_t$ and $\tilde\psi_t$ converge to the identity
as $t\to\infty$ uniformly with respect to the spherical metric.
Therefore, the distance between $\chi_t$ and $\tilde\chi_t$
in $C(X,\CP^1)$ tends to 0 as $t\to\infty$.

\begin{prop}
\label{P:Gamma}
There is a continuous map $\Gamma:[t_0,\infty)\times[0,1]\to\Cc$,
where $t_0$ is a sufficiently large real number,
such that
$$
\Gamma(t,0)=\chi_t,\quad \Gamma(t,1)=\tilde\chi_t,
$$
and the length of the path $s\mapsto \Gamma(t,s)$, $s\in [0,1]$
in $\Cc$ tends to $0$ as $t\to\infty$.
Moreover, for every $t\in [t_0,\infty)$, there exists a neighborhood
$O_t$ of $P_f$ such that $\Gamma(t,s)\in\Cc(O_t)$ for all $s\in [0,1]$.
For every finite interval $[t_1,t_2]$, there exists a neighborhood $O$ of $P_f$
that is contained in $O_t$ for all $t\in [t_1,t_2]$. 
\end{prop}

\begin{proof}
It suffices to define a continuous map $\Psi:[t_0,\infty)\times[0,1]\to C(\CP^1,\CP^1)$
such that $\Psi(t,0)=\psi_t$, $\Psi(t,1)=\tilde\psi_t$, and $s\mapsto\Psi(t,s)$ is a rectifiable path in
$C(\CP^1-P_t,Y)$, whose length tends to 0 as $t\to\infty$, and such that all
$\Psi(t,s)$ are holomorphic on some fixed neighborhood $\tilde W$ of $P_R$.
Then we set $\Gamma(t,s)=\Psi(t,s)\circ\phi_t$.

Fix $x\in P_f$, and set $z_t=\phi_t(x)$.
Let $W_x$ be the component of $W$ containing $\iota(x)$, and
$\xi$ a holomorphic coordinate on $W_x$ such that $\xi(\iota(x))=0$.
We have $z_t\in W_x$ for all sufficiently large $t$, and $\xi(z_t)\to 0$ as $t\to\infty$.
Since both $\xi\circ\psi_t(z)$ and $\xi\circ\tilde\psi_t(z)$ have simple
zeros at $z_t$, the ratio $\xi\circ\tilde\psi_t/\xi\circ\psi_t$ extends holomorphically to $W_x$
and converges to 1 on $W_x$ as $t\to\infty$.
It follows that the distance between $\psi_t$ and $\tilde\psi_t$ in 
$C(W_x-\{z_t\},Y)$ tends to zero as $t\to\infty$.
We set
$$
\xi\circ\Psi(t,s)=(\xi\circ\psi_t)\cdot
\exp\left(s\log\frac{\xi\circ\tilde\psi_t}{\xi\circ\psi_t}\right),
$$
on a neighborhood $\tilde W_{x}$ of $\iota(x)$ such that the right-hand side lies in $\xi(W_x)$
(we can choose one neighborhood that will work for all sufficiently large $t$). 
The branch of the logarithm is chosen to be the closest to 0.
This defines the maps $\Psi(t,s)$ on $\tilde W_{x}$.
Note that the path $s\mapsto\Psi(t,s)$ is rectifiable in $C(\tilde W_{x}-\{z_t\},Y)$,
and the length of the path is equal to the distance between $\psi_t$ and
$\tilde\psi_t$ in $C(\tilde W_{x}-\{z_t\},Y)$. 
The same formula defines the maps $\Psi(t,s)$ on a neighborhood of any point from $P_t$. 
Clearly, we can extend $\Psi(t,s)$ to $\CP^1$ with desired properties.
\end{proof}

In the proof of Theorem \ref{T:conv}, we will need two more lemmas.

\begin{lem}
  \label{L:ext}
  Consider a rectifiable path $\delta:[0,1]\to\Cc(O)$.
  Then there is a continuous extension $\delta:[0,\infty]\to\Cc(O)$
  such that $R\circ\delta(t+1)=\delta(t)\circ f$ for all $t\in [0,\infty)$.
  Moreover, the length of the extended path $\delta$ is at most $E/(E-1)$ times
  the length of the original $\delta$, and we have
  $$
  R\circ\delta(\infty)=\delta(\infty)\circ f.
  $$
\end{lem}

\begin{proof}
  Consider the lift $\alpha$ of the path $\delta$ as in Lemma \ref{L:lifting}.
  By Remark \ref{R:lift-contr}, the length of $\alpha$ is at most $E^{-1}$ times the length of $\delta$.
  Now we set $\delta(t)=\alpha(t-1)$ for $t\in [1,2]$.
  As we keep doing this extension process, we obtain more and more segments of
  $\delta$, each segment being shorter than the preceding one by at least the
  factor $E^{-1}$.
  It follows that $\delta(t)$ converges in $C(X-P_f,Y)$ as $t\to\infty$.
  Denote the limit by $\delta(\infty)$.
  Thus we obtain the extended path $\delta:[0,\infty]\to C(X-P_f,Y)$.
  The length of this extended path can be estimated by a geometric series with the common ratio $E^{-1}$:
  it does not exceed $E/(E-1)$.

  It remains to prove that $\delta(\infty)\in\Cc(O)$.
The map $\delta(\infty)$ is holomorphic on $O$ as a uniform limit of holomorphic maps.
  The only non-obvious property is that the preimage of $R^{-1}(P_R)$ under
  $\delta(\infty)$ is contained in $f^{-1}(P_f)$, or, equivalently,
  the image of $X-f^{-1}(P_f)$ under $\delta(\infty)$ is contained in $\tilde Y$.
  Indeed, the lift $\tilde\delta$
  of the path $t\mapsto \delta(t)$, $t\in [0,\infty]$,
  such that $\tilde\delta(0)=\delta(1)$ is unique.
  Therefore, we must have $\tilde\delta(t)=\delta(t+1)$ on $X-f^{-1}(P_f)$ for all $t\in [0,\infty]$,
in particular, $\tilde\delta(\infty)=\delta(\infty)$.
  On the other hand, by construction, the map $\tilde\delta(\infty)$ takes the 
set $X-f^{-1}(P_f)$ to $\tilde Y$.
Therefore, we have
  $\delta(\infty)(X-f^{-1}(P_f))\subseteq\tilde Y$, as desired.
  The equality $R\circ\delta(\infty)=\delta(\infty)\circ f$ follows from the equality
  $R\circ\tilde\delta(\infty)=\delta(\infty)\circ f$ on $X-f^{-1}(P_f)$.
\end{proof}

\begin{lem}
\label{L:not-too-close}
 There exists a real number $\eps>0$ such that the distance between 
two maps $\chi^*$, $\chi^{**}\in\Cc$ is bigger than $\eps$ provided that 
$$
R\circ\chi^*=\chi^*\circ f,\quad R\circ\chi^{**}=\chi^{**}\circ f.
$$
\end{lem}

\begin{proof}
There exists a neighborhood $O$ of $P_f$ such that every map $\chi\in\Cc$
such that $R\circ\chi=\chi\circ f$ is holomorphic on $O$.
Indeed, $\chi$ is holomorphic on at least some small neighborhood of $P_f$
but the formula $R\circ\chi=\chi\circ f$ says that $\chi$ is also holomorphic
on iterated pullbacks of this neighborhood under $f$. 

 Let $\xi$ be an extended B\"ottcher coordinate on a neighborhood $W$ of $P_R$ so that 
$\xi\circ R(z)=\xi(z)^{\nu(z)}$ for some locally constant function $\nu$
taking positive integer values.
Define the Green function $G_R$ of $R$ on $W$ as $-\log|\xi|$.
The function $G_R\circ\chi$ restricted to some neighborhood of $P_f$ is the same
for all maps $\chi\in\Cc$ with the property $R\circ\chi=\chi\circ f$
(by the uniqueness of B\"ottcher's coordinate).
Denote this function by $G_f$ and call it the Green function of $f$.
The Green function $G_R$ can be extended to $\CP^1$ by setting 
$$
G_R(z)=\frac 1{d^n}G_R(R^{\circ n}(z))
$$
if $R^{\circ n}(z)\in W$ for some $n\ge 0$, and $G_R(z)=0$ otherwise.
Here $d=\deg(R)$.
Similarly, $G_f$ extends to a function on $X$.

Take a sufficiently small number $\eps_0>0$.
The set $G_f\le\eps_0$ is mapped to the set $G_R\le\eps_0$ by both $\chi^*$ and $\chi^{**}$.
Note that $\xi\circ\chi^*$ and $\xi\circ\chi^{**}$ can only differ on
the set $G_f\le\eps_0$ by a locally constant factor that 
is a root of unity of bounded degree.
Therefore, the distance between restrictions of $\xi\circ\chi^*$ and $\xi\circ\chi^{**}$
to the set $G_f\le\eps_0$ cannot take arbitrarily small nonzero values.
It follows that if $\chi^*$ and $\chi^{**}$ are sufficiently close, then their 
restrictions to the set $G_f\le\eps_0$ must coincide.

Since all critical values of $R$ are poles of the Green function $G_R$,
the distance between two different $R$-preimages of any point $z$ with
$G_R(z)\ge\eps_0/d$ is bounded below by some positive number uniform with respect to $z$.
It follows that $\chi^*$ and $\chi^{**}$ must also coincide on the set 
$\eps_0\le G_f\le d\eps_0$ provided that $\chi^*$ and $\chi^{**}$ are sufficiently close.
By induction, the two maps coincide on the set $G_f>0$.
This set is dense in $X$ (because it is dense in $\CP^1-Z$), hence $\chi^*=\chi^{**}$.
\end{proof}

\begin{proof}[Proof of Theorem \ref{T:conv}]
 Throughout the proof, the parameters $t$ and $s$ will run through the interval $[0,1]$.
By a homotopy, we will always mean a homotopy between two paths in the metric space
$\Cc$ with fixed endpoints.
  
Consider any rectifiable path $\gamma:[0,1]\to\Cc$ connecting $\chi_{n-1}$
  with $\chi_{n}$ and homotopic to the path $t\mapsto \chi_{n-1+t}$.
  Consider the lift $\tilde\gamma:[0,1]\to\Cc$ of $\gamma$ (as in Lemma \ref{L:lifting})
  such that $\tilde\gamma(0)=\tilde\chi_n$.
  The path $\tilde\gamma$ is at least $E$ times shorter than $\gamma$.
  We claim that $\tilde\gamma(1)=\tilde\chi_{n+1}$.
  Indeed, since the path $\gamma$ is homotopic to the path $t\mapsto \chi_{n-1+t}$,
  the path $\tilde\gamma$ is homotopic to the path
  $t\mapsto\tilde\chi_{n+t}$ (the lifts of two homotopic paths are homotopic).

  Set $L_{n}$ to be the infimum of the lengths of rectifiable paths in $\Cc$ 
  lying in $\Cc(O)$ for some neighborhood $O$ of $P_f$, connecting
  $\chi_{n}$ to $\chi_{n+1}$ and homotopic to the path $t\mapsto \chi_{n+t}$.
  If $\eps_n$ denotes the maximum of the lengths of the paths
  $\Gamma_{n}:s\mapsto \Gamma(n,s)$ and $\Gamma_{n+1}:s\mapsto\Gamma(n+1,s)$, then we have
  $$
  L_{n}\le E^{-1}\cdot L_{n-1}+2\eps_{n}
  $$
  Indeed, if $\gamma$ is a rectifiable path in $\Cc$ that connects $\chi_{n-1}$
  with $\chi_{n}$, then $L_{n}$ is at most the length of the composition of
  the following paths:
  \begin{itemize}
    \item the path $\Gamma_{n}$ from $\chi_{n}$ to
    $\tilde\chi_{n}$;
    \item the path $\tilde\gamma$ from $\tilde\chi_{n}$
    to $\tilde\chi_{n+1}$;
    \item the reversed path $\Gamma_{n+1}$ from $\tilde\chi_{n+1}$
    to $\chi_{n+1}$.
  \end{itemize}
  The length of this composition can be made smaller than any fixed
  number exceeding $E^{-1}L_{n-1}+2\eps_{n}$ by choosing the path $\gamma$
in $\Cc(O)$ for some $O$
  to have length close to $L_{n-1}$.
  Note also that the composition is homotopic to $t\mapsto\chi_{n+t}$
  provided that $\gamma$ is homotopic to $t\mapsto\chi_{n-1+t}$.
  The corresponding homotopy can be easily constructed using the homotopy $\Gamma$.

  Take any $n_0$, then, applying the previous inequality several times,
  we obtain
  $$
  L_n\le E^{n_0-n}L_{n_0}+2\tilde\eps_{n_0}(q^{n-n_0}+q^{n-2}+\dots+1),
  $$
  where $\tilde\eps_{n_0}$ is the supremum of $\eps_{n_0+1}$, $\dots$.
  The second term in the right-hand side can be made arbitrarily
  small (uniformly with $n$)
  by choosing $n_0$ large enough.
  After $n_0$ has been chosen, we can choose sufficiently large $n$
  to make the first term as small as we wish.
  It follows that $L_n\to 0$ (in particular, the distance
  between $\chi_{n}$ and $\chi_{n+1}$ tends to 0 in $\Cc$).

  Consider the composition $\delta_n$ of some path $\gamma$ of length at most $2L_n$ homotopic
  to $t\mapsto\chi_{n+t}$ and the path $\Gamma_{n+1}$.
  Reparameterize $\delta_n$ so that the parameter runs from 0 to 1.
We can arrange that $\delta_n(t)\in\Cc(O)$ for some open neighborhood $O$ of $P_f$
and all $t\in[0,1]$.
  We have $R\circ\delta_n(1)=\delta_n(0)\circ f$ because
  $\delta_n(0)=\chi_{n}$ and $\delta_n(1)=\tilde\chi_{n+1}$.

  Consider the extended path $\delta_{n}:[0,\infty]\to\Cc(O)$
  as in Lemma \ref{L:ext}.
  Then we have
  $$
  R\circ\delta_n(\infty)=\delta_n(\infty)\circ f.
  $$
  The distance between $\delta_{n}(\infty)$ and $\delta_{m}(\infty)$ tends
  to 0 as $n$ and $m\to\infty$.
  By Lemma \ref{L:not-too-close},
  the sequence $\delta_{n}(\infty)$ stabilizes, i.e. $\delta_{n}(\infty)$
  is the same map $\chi_\infty$ for all sufficiently large $n$.
  We know that the distance
  between $\chi_n$ and $\chi_\infty$ in $\Cc$ tends to $0$.
  Therefore, $\chi_n$ converge to $\chi_\infty$.
\end{proof}

\end{document}